\documentclass[10pt]{amsart}
\usepackage{ulem}
\usepackage{array,youngtab}
\usepackage{color, graphicx, comment}
\usepackage{tikz}
\usepackage{ulem}
\usetikzlibrary{automata,positioning}

\usepackage{mathtools}

\definecolor{darkgreen}{rgb}{0.,0.5,0.}

\DeclareMathOperator{\CF}{CF}

\allowdisplaybreaks

\numberwithin{equation}{section} \overfullrule 5pt
\newtheorem{thm}{Theorem}[section]
\newtheorem{prop}[thm]{Proposition}

\newtheorem{lem}[thm]{Lemma}

\theoremstyle{definition}

\newtheorem{remark}{Remark}

\newcommand{\ff}{\mathbb{F}_2}

\title[
Period-doubling Continued Fractions are Algebraic in Characteristic $2$
]
{
Period-doubling Continued Fractions are  Algebraic in Characteristic $2$}
\date{}

\author{Yining Hu}
\address[Yining Hu, corresponding author]{School of Mathematics and Statistics, 
Huazhong University of Science and Technology, Wuhan, PR China}
\email{huyining@protonmail.com}

\author{Alain Lasjaunias}
\address[Alain Lasjaunias, (ex-member)]{
     Mathematical Institute Bordeaux University, Bordeaux, France.
}
\email{lasjauniasalain@gmail.com}

\subjclass[2010]{11B85, 11J70, 11B50, 11Y65, 05A15, 11T55}

\keywords{algebraicity, automatic sequence, continued fraction, period-doubling sequence}

\begin{document}
\begin{abstract}
	Considering an arbitrary pair of distinct and non constant polynomials, 
	$a$ and $b$ in $\mathbb{F}_2[t]$, we build a continued fraction in 
	$\mathbb{F}_2((1/t))$ whose partial quotients are only equal to $a$ or $b$.
	In a previous work of the first author and Han (to appear in Acta Arithmetica), the authors considered two 
	cases where the sequence of partial quotients represents in each case a 
	famous and basic $2$-automatic sequence, both defined in a similar way by 
	morphisms. They could prove the algebraicity of the corresponding continued 
	fractions for several pairs $(a,b)$ in the first case (the Prouhet-Thue-Morse
	sequence) and gave the proof for a particular pair for the second case (the period-doubling 
	sequence). Recently Bugeaud and Han (arXiv:2203.02213) proved the 
	algebraicity for an arbitrary pair in the first case.
	Here we give a short proof for an arbitrary pair in the second 
	case.
\end{abstract}

\maketitle

\section{Introduction}
	Let $\mathbb{F}_2((1/t))$ be the 
	field of power series in $1/t$, where t is a formal indeterminate, over the finite 
	field $\mathbb{F}_2$. A non-zero element is $\alpha=t^n+\sum_{i<n} a_i\cdot t^i$ 
	where $n$ belongs to $\mathbb{Z}$ and $a_i=0$ or $1$. An absolute value on $\mathbb{F}_2((1/t))$ is 
	defined by $|0|=0$ and $|\alpha|=|t|^n$ where $|t|>1$ is a fixed given real 
	number .

	So $\mathbb{F}_2((1/t))$ is the completion of $\mathbb{F}_2(t)$ for this absolute value. Every 
	irrational (resp. rational) element in $\mathbb{F}_2((1/t))$ can be expanded in an infinite (resp. finite)
	continued fraction: $\alpha=[a_0;a_1,...,a_n,...]$ where the $a_i$ are in 
	$\mathbb{F}_2[t]$ and $\deg(a_i)>0$ for $i >0$. For a basic introduction on fields of 
	power series and continued fractions, the reader may consult 
	\cite{lasjaunias2017}.

  The origin of the question discussed here is due to G.-N. Han and the first
	author. 
	In \cite{Hu2022H}, the authors considered two basic sequences, formed by an 
	infinite word with two letters $a$ and $b$, belonging to the family of 
	$2$-automatic sequences (see \cite{Allouche2003Sh} p.173 and p.176) . Both 
	sequences are obtained in a similar way, as fixed point of a morphism. For 
	the first one, the $(a;b)$-Prouhet-Thue-Morse sequence, denoted by $\mathbf{t}$,
  the morphism $\tau$ is defined by $\tau(a)=ab$ and $\tau(b)= ba$, and we have

     $$\mathbf{t}=\tau^{\infty}(a) = abbabaabbaab\ldots$$

  Note that this famous sequence was considered very long ago and has been the 
	starting point of various studies. For the second one , the 
	$(a;b)$-period-doubling sequence, denoted by $\mathbf{p}$, the morphism $\sigma$ is 
	defined by $\sigma(a)=ab$ and $\sigma(b)=aa$, and we have

     $$\mathbf{p}=\sigma^{\infty} (a) = abaaabababaaabaa\ldots$$

  In both cases, we consider a pair $(a,b)$ of distinct and non-constant 
	polynomials in $\mathbb{F}_2[t]$, and we can associate with each sequence an infinite 
	continued fraction in $\mathbb{F}_2((1/t))$, $\CF(\mathbf{t})$ and $\CF(\mathbf{p})$, where the sequence of 
	partial quotients is derived from the sequences $\mathbf{t}$ and $\mathbf{p}$ :  
	$$\alpha_\mathbf{t}=\CF(\mathbf{t})=[a;b,b,a,b,a,a,\ldots]$$ and 
	$$\alpha_\mathbf{p}=\CF(\mathbf{p})=[a;b,a,a,a,b,a,\ldots]. $$
 
	It was proved in \cite{Hu2022H} that, for pairs $(a,b)$ with 
	$\deg a +\deg b\leq 7$, $\alpha_\mathbf{t}$ was a root of a polynomial of degree $4$, with five 
  coefficients in $\mathbb{F}_2[t]$ depending on $(a,b)$. In a recent work, extending 
	the case of $\CF(\mathbf{t})$, Bugeaud and Han \cite{Bugeaud2022H} 
  obtained the same result for all pairs $(a,b)$ and they gave the explicit 
	formulas for the $5$ coefficients as polynomials in $a$ and $b$.

	Also in \cite{Hu2022H}, it was proved that for $(a,b)=(t^3,t^2+t+1)$  we have 

	$$ (t^5+t^3+t^2)\cdot\alpha_\mathbf{p}^4+(t^8+t^6+t^5+t^3)\cdot\alpha_\mathbf{p}^3+(t^5+t^4+t^3)\cdot\alpha_\mathbf{p}^2+1=0.$$ 

  In this note, we give a short proof of the general case for the sequence $\mathbf{p}$ 
	with the following theorem, thus  confirming Conjecture 1.5 from 
	\cite{Hu2022H}.

\begin{thm}\label{thm:main}
	Let $a,b$ be two distinct non constant elements in $\mathbb{F}_2[t]$.
	Let $\alpha_\mathbf{p}=[\sigma^\infty(a)]\in \mathbb{F}_2((1/t))$. 
	Define $P(x)\in \mathbb{F}_2(t)[x]$ to be
\begin{equation}\label{eq:pd}
	P(x)=Ax^4+Bx^3+Cx^2+1
\end{equation}
with
\begin{equation*}
	A=ab+b^2+1,\quad
	B=ab(a+b),\quad
	C=ab.
\end{equation*}
	Then $P(\alpha_\mathbf{p})=0$.
\end{thm}
\begin{remark}
	An elementary proof shows that $P(x)$ has no solution in $\ff(t)$. Consequently,
	since $\alpha_\mathbf{p}$ is not quadratic, $P(x)$ is irreducible and 
	$\alpha_\mathbf{p}$ has degree $4$ over $\ff(t)$.
\end{remark}
\begin{remark}
	The theorem remains true when we replace $\mathbb{F}_2$ by any other field
	$K$ of characteristic $2$.
\end{remark}
\begin{remark}
	In fact we have proven that the continued fraction $\alpha_\mathbf{p}$ as a series in 
	the ring $\mathbb{F}_2((1/a,1/b))$ is algebraic over $\mathbb{F}_2(a,b)$.
\end{remark}

	In Section \ref{sec:notation} we recall notation and formulas for
	continued fractions.
 In Section \ref{sec:proof} we give the proof of Theorem \ref{thm:main} 
 and in a last section we make some comments about the link with Riccati
 differential equations.

\section{Notation and basic formulas for continued fractions}
\label{sec:notation}
We use the same notation as in \cite{lasjaunias2017}, which we recall in this 
section.

Let $W=w_1,w_2,\ldots,w_n$ be a sequence of variables over a ring $\mathbb{A}$.
We set $|W|=n$ for the length of the word $W$. We define the following operators
for the word $W$.
\begin{align*}
	W'&=w_2,w_3,\ldots, w_n \mbox{ or } W'=\emptyset \mbox{ if }|W|=1.\\
	W''&=w_1,w_2,\ldots, w_{n-1} \mbox{ or } W''=\emptyset\mbox{ if }|W|=1.\\
  W^*&=w_n,w_{n-1},\ldots,w_1.
\end{align*}
We consider the finite continued fraction associated to $W$ to be
\begin{equation}\label{eq:fncf}
	[W]=[w_1,\ldots, w_n]=w_1+\cfrac{1}{w_2+\cfrac{1}{w_3+\cfrac{1}{\ddots+\cfrac{1}{w_n}}}}
\end{equation}
The continued fraction $[W]$ is a quotient of multivariate polynomials in the variables
$w_1,w_2,\ldots,w_n$. These polynomials are called continuants
built on $W$. They are defined inductively as follows:

Set $\langle \emptyset\rangle =1$. If the sequence $W$ has only one element, then we have
$\langle W\rangle =W$. Hence, with the above notation, the continuants can be computed, 
recursively on the length $|W|$, by the following formula
\begin{equation}\label{eq:recur}
	\langle W\rangle =w_1\langle W'\rangle +\langle (W')'\rangle  \mbox{ for } |W|\geq 2.
\end{equation}
Thus, with this notation, for any finite word $W$, the finite continued fraction
$[W]$ satisfies
\begin{equation}\label{eq:frac}
[W]=\frac{\langle W\rangle }{\langle W'\rangle }.
\end{equation}
It is easy to prove by induction that $\langle W^*\rangle = \langle W\rangle$.
For any finite sequences $A$ and $B$ of variables over $\mathbb{A}$, defining
$A,B$ as the concatenation of the sequences, by induction on $|A|$, 
we also have the following generalization of \eqref{eq:recur}
\begin{equation}\label{eq:recur2}
\langle A,B\rangle =\langle A\rangle \langle B\rangle +\langle A''\rangle \langle B'\rangle .
\end{equation}
Using induction of $|W|$, we have the following classical identity
\begin{equation}\label{eq:pq}
	\langle W\rangle \langle (W')''\rangle -\langle W'\rangle \langle W''\rangle =(-1)^{|W|} \mbox{ for } |W|\geq 2.
\end{equation}


\section{Proof of Theorem \ref{thm:main}}\label{sec:proof}

 For $n\geq 0$, set 
$$W_n=(\sigma^n(a))''.$$
We have $W_0=\emptyset$, $W_1=a$, $W_2=aba$, etc., and $|W_n|=2^n-1$.

We will prove that $P([W_n])$ converges to $0$ .
For this, we need to following two lemmas.
\begin{lem}\label{lem:sym}
	For all $n\geq 0$, 
	\begin{equation}\label{eq:sym}
	W_{n+1}=W_n,\varepsilon_n, W_n 
	\end{equation}
	where $\varepsilon_n=a$ if $n$ is even and $\varepsilon_n=b$ if $n$ is odd.
	In consequence, for $n\geq 0$, we have
	$$W_n=W_n^*.$$
\end{lem}
\begin{proof}
	We prove by induction. Suppose identity \eqref{eq:sym} holds for $n$, then
	\begin{align*}
		W_{n+2}&= \sigma(\sigma^{n+1}(a))''\\
		&=\sigma(W_{n+1})a\\
		&=\sigma(W_n),\sigma(\varepsilon_n),\sigma(W_n)a\\
		&=W_{n+1},\varepsilon_{n+1},W_{n+1}.
	\end{align*}
	In the last equality, we use the fact that $\sigma(\varepsilon_n)=a
	\varepsilon_{n+1}$.
\end{proof}

For $n\geq 1$, define $u_n=\langle W_n \rangle$ and $v_n=\langle W_n'\rangle$,
so that $$[W_n]=\frac{u_n}{v_n}.$$
\begin{lem}\label{lem:uv}
	For all $n\geq 1$, we have
	\begin{align*}
		u_{n+1}&=\varepsilon_n u_n^2,\\
		v_{n+1}&=\varepsilon_n u_n v_n+1.\\
	\end{align*}
\end{lem}
\begin{proof}
	We use Lemma \ref{lem:sym}, identity \eqref{eq:recur2}, and the fact
	that we are in characteristic $2$.

	\begin{align*}
		u_{n+1}&=\langle W_{n+1}\rangle \\
		&=\langle W_n,\varepsilon_n,W_n\rangle \\
		&=\langle W_n\rangle \langle \varepsilon_n,W_n\rangle +\langle W_n''\rangle \langle W_n\rangle \\ 
		&=\langle W_n\rangle (\varepsilon_n \langle W_n\rangle +\langle W_n'\rangle )+\langle W_n''\rangle \langle W_n\rangle \\ 
		&=\langle W_n\rangle (\varepsilon_n \langle W_n\rangle +\langle W_n'\rangle )+\langle W_n'\rangle \langle W_n\rangle \\ 
		&=\langle W_n\rangle \varepsilon_n \langle W_n\rangle \\
		&=\varepsilon_n u_n^2.
	\end{align*}
	\begin{align*}
		v_{n+1}&=\langle W_{n+1}'\rangle \\
		&=\langle W_n',\varepsilon_n,W_n\rangle \\
		&=\langle W_n'\rangle \langle \varepsilon_n,W_n\rangle +\langle (W_n')''\rangle \langle W_n\rangle \\
		&=\langle W_n'\rangle (\varepsilon_n\langle W_n\rangle +\langle W_n'\rangle )+\langle (W_n')''\rangle \langle W_n\rangle \\
		&=\varepsilon_n u_nv_n+\langle W_n'\rangle \langle W_n'\rangle +\langle (W_n')''\rangle \langle W_n\rangle \\
		&=\varepsilon_n u_nv_n+1.
	\end{align*}
	In the last step, we also use \eqref{eq:pq}, and the symmetry of $W_n$. \end{proof}
\begin{proof}[Proof of Theorem \ref{thm:main}]
	For $n\geq 1$, set $$X_n=Au_n^4+Bu_n^3 v_n+Cu_n^2v_n^2+v_n^4,$$
	so that $P(u_n/v_n)=X_n/v_n^4$.

	Using Lemma \ref{lem:uv} and noticing that 
	$\varepsilon_n^2+ab=\varepsilon_n(a+b)$ , we obtain 
	\begin{equation}\label{eq:x}
	X_{n+1}+1=\varepsilon^4_n u_n^4 (X_n+1)+\varepsilon_n^3u_n^4(a+b)(1+abu_n^2).
	\end{equation}
	We observe taht 
	$$X_1+1=a^3(a+b)=\varepsilon_0u_1^2(a+b),$$
	this allows us, using \eqref{eq:x} and the fact that $\varepsilon_n\varepsilon_{n-1}=ab$,
	to get by induction that for all $n\geq 1$
	\begin{equation}\label{eq:x2}
		X_n+1=\varepsilon_{n-1}u_n^2(a+b).
	\end{equation}
	Therefore 
	$$P\left(\frac{u_n}{v_n}\right)=\frac{\varepsilon_{n-1}u_n^2(a+b)+1}{v_n^4}$$
	This is a quotient of polynomials in $\mathbb{F}_2[t]$, and it can be easily seen from
	Lemma \ref{lem:uv} that the degree in $t$ of the denominator minus that of
	the numerator tends to plus infinity as $n$ goes to infinity.
	And therefore the sequence converges to $0$ in $\mathbb{F}_2((1/t))$.
\end{proof}

\section{Link with Riccati differential equations}\label{sec:last}
In both cases, the Prouhet-Thue-Morse  and the period-doubling
continued fraction, a particular choice of the pair $(a,b)$ in $\ff[t]^2$, 
brings us back to an investigation undertaken 45 years ago. Indeed, in these two
cases, taking $(a,b)=(t,t+1)$, all the partial quotients of these continued 
fractions are $t$ or $t+1$. In 1977, Baum and Sweet \cite{Baum1977S}
considered all the infinite continued fractions in 
$\mathbb{F}_2((1/t))$ whose partial quotients have all degree one.

They considered the subset $D$ of irrational continued fractions $\alpha$ in 
$\mathbb{F}_2((1/t))$  such that
    $$\alpha=[0;a_1,a_2,\ldots,a_n,\ldots]$$ 
where $\deg(a_i)=1$ (i.e., $a_i=t$ or $t+1$) for $i\geq 1$. 
Let $P$ the subset of $\mathbb{F}_2((1/t))$ containing all $\alpha$ such that 
$|\alpha|<1$. 
Hence $P$ contains $D$. We have :

\begin{thm}\label{thm:bs}[Baum-Sweet, 1977]
	An element $\alpha\in P$ is in $D$ if and only if $\alpha$ satisfies
	$$\alpha^2+t\alpha+1=(1+t)\beta^2$$
	for some $\beta\in P$.
\end{thm}

In order to obtain another formulation of Theorem \ref{thm:bs}, we will now use
formal differentiation in $\mathbb{F}_2((1/t))$. Continued fractions and formal 
differentiation in power series fields are related, these are the main tools in
Diophantine approximation (see \cite{lasjaunias2000}, p.221-225). We recall that
in characteristic $2$, which is the case considered here, an element has 
derivative zero if and only if this element is a square. Moreover the derivative
of an element is a square (hence the second derivative is zero) . 
We have the following theorem.

\begin{thm}\label{thm:3}
	An element $\alpha\in P$ is in $D$ if and only if $\alpha$ satisfies
	$$(\alpha\cdot t(t+1))'=\alpha^2+1.$$
\end{thm}
\begin{proof}
	If $$\alpha^2+t\alpha+1=(1+t)\beta^2$$ 
	then by derivation we get
	$$(t\alpha)'=\beta^2.$$
	Consequently, we have 
	$$\alpha^2+1=t\alpha+(1+t)(t\alpha)'=(\alpha\cdot t(t+1))'.$$
	Conversely, if 
	$$(\alpha\cdot t(t+1))'=\alpha^2+1.$$
	then 
	$$
	\alpha^2+t\alpha+1= (\alpha\cdot t(t+1))'+t\alpha=(t+1)(\alpha\cdot t)'
	=(t+1)\beta^2,
	$$
	for some $\beta\in P$ since $(t\alpha)'$ is a square in $P$.
\end{proof}
Thus the Riccati differential equation
\begin{equation}\label{eq:ric}
	(x\cdot t(t+1))'=x^2+1 \tag{R0}
\end{equation}
is characteristic of infinite continued fraction expansions in $\mathbb{F}_2((1/t))$,
having all partial quotients of degree one.
Indeed, there exists a direct proof, without Theorem \ref{thm:bs}, showing
that a solution in $\mathbb{F}_2((1/t))$ of \eqref{eq:ric} is irrational and has
all partial quotients of degree one (see \cite{lasjaunias2000}, p. 225).

Now we turn to elements in $\mathbb{F}_2((1/t))$ which are algebraic over $\ff(t)$ of degree $d\geq 1$.
If $\alpha$ is such an element, we have (see \cite{lasjaunias2000}, p. 221)
\begin{equation}
	\alpha'= a_0+a_1\alpha+\cdots+a_{d-1}\alpha^{d-1}
\end{equation}
where the coefficients $a_i$ belong to $\ff(t)$.

Concerning the algebraic continued fraction $\alpha_\mathbf{t}=\CF(\mathbf{t})\in \mathbb{F}_2((1/t))$
associated with the Prouhet-Thue-Morse sequence, we have observed that it
satisfies the Riccati differential equation (see \cite{Bugeaud2022H},
Proposition 2.4):
\begin{equation}\label{eq:ric2}
	(ab(a+b) x)'=(ab)'(1+x^2).\tag{R}
\end{equation}
In fact, the continued fraction $\CF(\mathbf{p})$ (also algebraic of degree
$4$), associated with the period-doubling sequence satisfies the same 
Riccati differential equation.

\begin{prop}
	The continued fraction $\alpha_\mathbf{p}=\CF(\mathbf{p})$ satisfies the Riccati differential
	equation \eqref{eq:ric2}
\end{prop}
\begin{proof}
	We know from Theorem \ref{thm:main} that
\begin{equation}\label{eq:pd}
	A \alpha_\mathbf{p}^4+B\alpha_\mathbf{p}^3+C\alpha_\mathbf{p}^2+1=0
\end{equation}
with
\begin{equation*}
	A=ab+b^2+1,\quad
	B=ab(a+b),\quad
	C=ab.
\end{equation*}
Hence, by straightforword derivation of \eqref{eq:pd}, we get 
$$
	B\alpha_\mathbf{p}'+B'\alpha_\mathbf{p}=A'\alpha_\mathbf{p}^2+C'.
$$
	Therefore, observing that $A'=C'=(ab)'$, we obtain
$$
	(ab(a+b) \alpha_\mathbf{p})'= (B\alpha_\mathbf{p})'
	=A'\alpha_\mathbf{p}^2+C'
	=(ab)'(1+\alpha_\mathbf{p}^2). \qedhere
	$$
\end{proof}
Note that for the pair $(a,b)$=$(t,t+1)$, equation 
\eqref{eq:ric2} reduces to \eqref{eq:ric}, the one stated in Theorem 
\ref{thm:3}. 
This was the way, taking in consideration differential aspects for $\CF(\mathbf{t})$ 
and $\CF(\mathbf{p})$, knowing the equation \eqref{eq:ric} in the basic case, 
and inspired by the particular case published in \cite{Hu2022H}, 
that allowed us to guess the coefficients for the algebraic equation satisfied 
by $\CF(\mathbf{p})$.

We have observed that the two particularly simple sequences, 
the Prouhet-Thue-Morse sequence and the period-doubling sequence, are close in many different ways, and not only in their generation process. 
This leads to the intuition that other $2$-automatic sequences, defined
on an alphabet of two letters, could have a representation as an algebraic continued fraction in $\mathbb{F}_2((1/t))$ (not necessarily of degree $4$), as it happens for the two famous ones considered up to now.

\bibliographystyle{plain}

\bibliography{article}

\begin{thebibliography}{1}

\bibitem{Allouche2003Sh}
Jean-Paul Allouche and Jeffrey Shallit.
\newblock {\em Automatic sequences. Theory, Applications, Generalizations}.
\newblock Cambridge University Press, Cambridge, 2003.

\bibitem{Baum1977S}
Leonard~E. Baum and Melvin~M. Sweet.
\newblock Badly approximable power series in characteristic {$2$}.
\newblock {\em Ann. of Math. (2)}, 105(3):573--580, 1977.

\bibitem{Bugeaud2022H}
Yann Bugeaud and Guo-Niu Han.
\newblock The {T}hue-{M}orse continued fractions in characteristic $2 $ are
  algebraic.
\newblock {\em arXiv preprint arXiv:2203.02213}, 2022.

\bibitem{Hu2022H}
Yining Hu and Guo-Niu Han.
\newblock On the algebraicity of {T}hue-{M}orse and period-doubling continued
  fractions.
\newblock {\em Acta Arithmetica}, 2022 (to appear).

\bibitem{lasjaunias2000}
Alain Lasjaunias.
\newblock A survey of diophantine approximation in fields of power series.
\newblock {\em Monatshefte f{\"u}r Mathematik}, 130(3):211--229, 2000.

\bibitem{lasjaunias2017}
Alain Lasjaunias.
\newblock Continued fractions.
\newblock {\em arXiv preprint arXiv:1711.11276}, 2017.

\end{thebibliography}

\end{document}